\documentclass[11pt]{amsart}
\usepackage{amscd}
\usepackage{amsmath, amssymb}
\usepackage{verbatim}
\usepackage{amsfonts}

\usepackage{amssymb,amsmath,amsthm,amscd,mathrsfs,graphicx}
\usepackage[cmtip,all]{xy}
\usepackage{pb-diagram}

\numberwithin{equation}{section}
\newtheorem{teo}{Theorem}[section]

\newtheorem{theorem}{Theorem}[section]

\newtheorem{lemma}{Lemma}[section]

\newcommand{\ve}{\varepsilon}

\theoremstyle{definition}

\theoremstyle{remark}
\newtheorem{remark}[teo]{Remark}

\begin{document}
\bibliographystyle{amsplain}

\title[Convexity of level sets]{Convexity of level sets \\ and a two-point function}
\author[B. Weinkove]{Ben Weinkove}
\address{Department of Mathematics, Northwestern University, 2033 Sheridan Road, Evanston, IL 60208}

\thanks{MSC 35J05, 31B05.  Keywords: convexity, level sets, harmonic functions, principal curvature.  Supported in part by National Science Foundation grant DMS-1406164.}
\maketitle

\vspace{-20pt}

\begin{abstract}
We establish a maximum principle for a two-point function in order to analyze the  convexity of level sets of harmonic functions.  We show that this can be used to prove a strict convexity result involving the smallest principal curvature of the level sets.
\end{abstract}

\section{Introduction} \label{intro}
 
The study of the convexity of level sets of solutions to elliptic PDEs has a long history, starting with the well-known result that the level curves of the Green's function of a convex domain $\Omega$ in $\mathbb{R}^2$ are convex \cite{A}.  In the 1950s Gabriel \cite{G} proved the analogous result in 3 dimensions and this was extended by Lewis \cite{L} and later Caffarelli-Spruck \cite{CS} to higher dimensions and  more general elliptic PDEs.  These results show that for a large class of PDEs, there is a principle that convexity properties of the boundary of the domain $\Omega$ imply convexity of the level sets of the solution $u$.

There are several approaches to these kinds of convexity results (see for example \cite[Section III.11]{Ka0}).  One is the ``macroscopic'' approach which uses a globally defined function of two points $x, y$ (which could be far apart) such as $u( \frac{1}{2}(x+y)) - \min(u(x), u(y))$.  Another is the ``microscopic'' approach which computes with functions of 
 the principal curvatures of the level sets at a single point.  This is often used together with a constant rank theorem.  There is now a vast literature on these and closely related results, see for example \cite{ALL, BG, BLS, Bo, BL, CF, CGM, DK, HNS, Ko1, Ko2, KL, RR, S, SWYY, SW, W}  and the references therein.

It is natural to ask whether these ideas can be extended to cases where the boundary of the domain is \emph{not} convex.  Are the level sets of the solution at least as convex as the boundary in some appropriate sense?  
In this short note we introduce a global ``macroscopic'' function of two points which gives a kind of measure of convexity and makes sense for non-convex domains.  Our function
\begin{equation}
(Du(y) - Du(x)) \cdot (y-x)
\end{equation}
 is evaluated at two points $x$, $y$ which are constrained to lie on the same level set of $u$. Under suitable conditions, a level set of $u$ is convex if and only if this quantity has the correct sign on that level set.
  We prove a maximum principle for this function using the method of Rosay-Rudin \cite{RR} who considered a different two-point function
\begin{equation} \label{RRf}
\frac{1}{2} (u(x)+u(y)) - u\left( \frac{x+y}{2}\right).
\end{equation}
In addition, we show that our ``macroscopic'' approach can be used to prove a ``microscopic'' result.  Namely, we localize our function and show that it gives another proof of a result of Chang-Ma-Yang \cite{CMY} on the principal curvatures of the level sets of a harmonic function $u$.
In this paper, we consider only the case of harmonic functions.  However, we expect that our techniques  extend to some more general types of PDEs.

We now describe our results more precisely.  Let $\Omega_0$ and  $\Omega_1$ be  bounded domains in $ \mathbb{R}^n$ with $\overline{\Omega}_1 \subset \Omega_0$.  
Define $\Omega= \Omega_0 \setminus \Omega_1$. 
Assume that $u \in C^1(\overline{\Omega})$ satisfies
\begin{equation}
\begin{split}
\Delta u = {} & 0 \quad \textrm{in } \Omega = \Omega_0 \setminus \overline{\Omega}_1 \\
u = {} & 0 \quad \textrm{on }\partial \Omega_0 \\
u ={} & 1 \quad \textrm{on } \partial \Omega_1,
\end{split}
\end{equation}
and 
\begin{equation} \label{grad}
Du \textrm{ is nowhere vanishing in $\Omega$}.
\end{equation}
It is well known that (\ref{grad}) is satisfied if $\Omega_0$ and $\Omega_1$ are both starshaped with respect to some point $p\in \Omega_1$.  A special case of interest is when both $\Omega_0$ and $\Omega_1$ are convex, but this is not required for our main result.

To introduce our two-point function, first fix  a smooth function $\psi : [0, \infty) \rightarrow \mathbb{R}$  satisfying
\begin{equation} \label{psiass}
\psi'(t) - 2|\psi''(t)| t \ge 0.
\end{equation}
For example, we could take $\psi(t) = at$ for $a \ge 0$.  
Then define
\begin{equation} \label{Qd}
Q(x,y) = (Du(y) - Du(x)) \cdot (y-x) + \psi( |y-x|^2)
\end{equation}
restricted to $(x,y)$ in 
$$\Sigma = \{ (x,y) \in \overline{\Omega} \times \overline{\Omega} \ | \ u(x)=u(y) \}.$$
Comparing with the Rosay-Rudin function (\ref{RRf}), note that the function $Q(x,y)$ does not require $(x+y)/2 \in \overline{\Omega}$ and makes sense
 whether or not $\partial \Omega_0$ or $\partial \Omega_1$ are convex.  Taking $\psi=0$, the level set $\{ u=c\}$ is convex if and only if the quantity $Q$  is nonpositive on $\{ u=c\}$.  If $\psi(t) = at$ for $a>0$ then $Q \le 0$ implies strict convexity of the level set.
More generally $Q$ gives quantitative information about the convexity of the level sets $\{ u=c \}$, relative to the gradient $Du$.

We also remark that the function (\ref{Qd}) looks formally similar to the two-point function of Andrews-Clutterbuck, a crucial tool in their proof of the fundamental gap conjecture \cite{AC}.  However, here $x$ and $y$ are constrained to lie on the same level set of $u$ and so the methods of this paper are quite different.

Our main result is:

\begin{theorem} \label{theoremmain} $Q$ does not attain a strict maximum at a point in the interior of $\Sigma$.
\end{theorem}

Roughly speaking, this result says that the level sets $\{ u=c\}$ for $0 \le c \le 1$ are ``the least convex'' when $c=0$ or $c=1$.
As mentioned above, the result holds even in the case that  $\partial \Omega_0$ and $\partial \Omega_1$ are non-convex. 

The proof of Theorem \ref{theoremmain} follows quite closely the paper of Rosay-Rudin \cite{RR}.  Indeed a key tool of \cite{RR} is Lemma \ref{lemmaRR} below which gives a map from points $x$ to points $y$ with the property that $x,y$ lie on the same level set.

Next we localize our function (\ref{Qd}) to prove a strict convexity result on the level sets of $u$.   If we assume  now that $\partial \Omega_0$ and $\partial \Omega_1$ are strictly convex, we can  apply the technique of Theorem \ref{theoremmain} to obtain an alternative proof of the following result of  Chang-Ma-Yang \cite{CMY}.

\begin{theorem} \label{theoremCMY}
Assume in addition that $\partial \Omega_0$ and $\partial \Omega_1$ are strictly convex and  $C^2$. Then the quantity $|Du| \kappa_1$ attains its minimum on the boundary of $\Omega$, where $\kappa_1$ is the smallest principal curvature of the level sets of $u$.
 \end{theorem}
 
Note that many other strict convexity results of this kind are proved in \cite{CMY, JMO, Lo, MOZ, MYY, OS, ZZ} and other papers using 
microscopic techniques.   

 The author thanks G. Sz\'ekelyhidi for some helpful discussions and the referee for useful comments.

\section{Proof of Theorem \ref{theoremmain}} \label{sectionharmonic}

First we assume that $n$ is even.   We suppose for a contradiction that $Q$ attains a maximum at an interior point, and assume that $\sup_{\Sigma} Q > \sup_{\partial \Sigma} Q$.  Then we may choose $\delta>0$ sufficiently small so that 
$$Q_{\delta} (x,y) = Q(x,y) + \delta |x|^2$$
still attains a maximum at an interior point.

We use a lemma from \cite{RR}.  Suppose $(x_0, y_0)$ is  an interior point with $u(x_0)=u(y_0)$.  We may assume that $D u(x_0)$ and $Du(y_0)$ are nonzero vectors.  Let $L$ be an element of $\textrm{O}(n)$ with the property  that 
\begin{equation} \label{Lc}
L(Du(x_0)) =  cDu(y_0),  \quad\textrm{for } c = | Du(x_0)|/ |Du(y_0)|.
\end{equation}
Note that there is some freedom in the definition of $L$.  We will make a specific choice later. 
Rosay-Rudin  show the following (it is a special case of \cite[Lemma 1.3]{RR}).

\begin{lemma} \label{lemmaRR}
There exists a real analytic function $\alpha(w) = O(|w|^3)$ so that for all $w \in \mathbb{R}^n$ sufficiently close to the origin,
\begin{equation} \label{RR}
u(x_0 + w) = u\big(y_0 + cLw + f(w) \xi + \alpha(w) \xi\big), \quad \textrm{where } \xi = \frac{Du(y_0)}{|Du(y_0)|},
\end{equation}
where $f$ is a harmonic function defined in a neighborhood of the origin in $\mathbb{R}^n$, given by
\begin{equation} \label{eqnf}
f(w) = \frac{1}{|Du(y_0)|} (u(x_0+w) - u(y_0+cLw)).
\end{equation}
\end{lemma}
\begin{proof}[Proof of Lemma \ref{lemmaRR}]  We include the brief argument here for the sake of completeness.  Define a real analytic map $G$ which takes $(w,\alpha) \in \mathbb{R}^n \times \mathbb{R}$ sufficiently close to the origin to 
$$G(w ,\alpha) = u\big(y_0 + cLw + f(w) \xi + \alpha \xi\big) - u(x_0+w),$$
for $c, L, \xi$  and $f$ defined by (\ref{Lc}), (\ref{RR}) and (\ref{eqnf}).  Note that $G(0,0)=0$ and, by the definition of $\xi$,
$$\frac{\partial G}{\partial \alpha} (0,0) =  D_i u (y_0) \xi_i = |Du(y_0)| >0,$$ where here and henceforth we are using the convention of summing repeated indices.

Hence by the implicit function theorem there exists a real analytic map $\alpha=\alpha(w)$ defined in a neighborhood $U$ of the origin in $\mathbb{R}^n$ to $\mathbb{R}$ with $\alpha(0)=0$ such that $G(w, \alpha(w))=0$ for all $w \in U$.  It only remains to show that $\alpha(w) = O(|w|^3)$.

Write $y = y_0 + cLw + f(w) \xi + \alpha(w) \xi$, $x = x_0+w$ and $L = (L_{ij})$ so that $L_{ij} D_ju(x_0) = c D_iu(y_0)$ and $c L_{ij} D_iu(y_0) = D_j u(x_0)$.  Then at $w \in U$,
\begin{equation} \label{dG}
\begin{split}
0 =  \frac{\partial G}{\partial w_j} 
%=  {} &  u_i(y) \frac{\partial}{\partial w_j} ((y_0)_i + c(Lw)_i + f(w)\xi_i + \alpha(w) \xi_i) - u_j (x) \\
= {} &  D_i u(y) \bigg( c L_{ij} + \frac{(D_ju (x) - c D_ku(y_0+cLw) L_{kj})}{|Du(y_0)|}  \xi_i \\ {} & + \frac{\partial \alpha}{\partial w_j} \xi_i \bigg) - D_ju(x),
\end{split}
\end{equation}
and evaluating at $w=0$ gives $0= |Du(y_0)|  \frac{\partial \alpha}{\partial w_j}(0)$
%$$0 = u_j(x_0) + u_j(x_0) - u_j(x_0)  + |Du(y_0)|  \frac{\partial \alpha}{\partial w_j}(0)- u_j(x_0)= |Du(y_0)|  \frac{\partial \alpha}{\partial w_j}(0),$$
and hence $\partial \alpha/\partial w_j(0) =0$ for all $j$.

Differentiating (\ref{dG}) and evaluating at $w=0$, we obtain for all $j, \ell$,
\[
\begin{split}
0 = {} &   \frac{\partial^2 G}{\partial w_{\ell} \partial w_j} \\
=  {} &   D_k D_i u(y_0) c^2 L_{ij} L_{k \ell}  - D_{\ell} D_j u(x_0) \\ {} & + D_i u(y_0) \bigg( \frac{(D_{\ell} D_j u(x_0) - c^2 D_m D_k u(y_0) L_{kj} L_{m\ell})}{|Du(y_0)|}  \xi_i  + \frac{\partial^2 \alpha}{\partial w_{\ell} \partial w_j}(0) \xi_i \bigg) \\
= {} &  | Du(y_0)| \frac{\partial^2 \alpha}{\partial w_{\ell} \partial w_j}(0).
\end{split}
\]
Hence $\alpha(w) = O(|w|^3)$, as required.
\end{proof}

Now assume that $Q_{\delta}$ achieves a maximum at the interior point $(x_0, y_0)$.  Write $x= x_0+w =(x_1, \ldots, x_n)$ and $y = y_0 + cLw + f(w) \xi + \alpha(w) \xi=(y_1, \ldots, y_n)$ and
$$F(w) = Q_{\delta}(x,y) = Q(x_0 + w, y_0 + cLw + f(w) \xi + \alpha(w) \xi) + \delta |x_0+w|^2.$$
To prove the lemma it suffices to show that $\Delta_w F(0) > 0$, where we write $\Delta_w = \sum_j \frac{\partial^2}{\partial w_j^2}$.   Observe that
$$\Delta_w x(0) = 0 = \Delta_w y(0).$$
Hence, evaluating at $0$, we get
\[
\begin{split}
\Delta_w F = {} & \sum_j ( \frac{\partial^2}{\partial w_j^2} (D_iu(y) - D_iu(x)) ) (y_i-x_i) \\ {} & + 2  \frac{\partial}{\partial w_j} (D_i u(y) - D_i u(x)) \frac{\partial}{\partial w_j} (y_i-x_i)  + \sum_j \frac{\partial^2}{\partial w_j^2} \psi( |y-x|^2)  + 2n\delta.
\end{split}
\]
First compute
\begin{equation*} \label{psicomp}
\begin{split}
\lefteqn{ \sum_j \frac{\partial^2}{\partial w_j^2} \psi( |y-x|^2)  } \\
%= {} & 2 \sum_j \frac{\partial}{\partial s_j} \left( \psi' \sum_i (y_i-x_i) \frac{\partial}{\partial s_j} (y_i-x_i) \right) \\
%= {} &2 \psi' \sum_{i,j} \left( \frac{\partial}{\partial w_j} (y_i-x_i) \right)^2  + 4 \psi''  \sum_j \left( \sum_i (y_i-x_i) \frac{\partial}{\partial w_j} (y_i-x_i) \right)  \left( \sum_k (y_k-x_k) \frac{\partial}{\partial w_j} (y_k-x_k) \right) \\
= {} & 2 \psi' \sum_{i,j} \left( c L_{ij} - \delta_{ij} \right)^2
+ 4 \psi'' \sum_j \left( \sum_i (y_i-x_i) (cL_{ij} - \delta_{ij}) \right)^2 \\ % \left( \sum_k (y_k-x_k) (cL_{kj} - \delta_{kj}) \right) \\
%\ge {} & 2 \psi' \sum_{i,j} \left( c L_{ij} - \delta_{ij} \right)^2 
%- 4 | \psi''| \sum_j \left( \sum_i (y_i-x_i)^2 \right) \left( \sum_i (cL_{ij}  - \delta_{ij})^2 \right) \\
\ge {} & 2 \psi' \sum_{i,j} \left( c L_{ij} - \delta_{ij} \right)^2 
- 4 | \psi''| | y-x|^2 \sum_{i,j} (cL_{ij}  - \delta_{ij})^2 \ge 0,
\end{split}
\end{equation*}
using the Cauchy-Schwarz inequality and the condition (\ref{psiass}).

Next, at $w=0$,
\[
\begin{split}
%\frac{\partial y_i}{\partial w_j}   =  {} & c L_{ij}, \quad \frac{\partial x_i}{\partial w_j} = \delta_{ij} \\
\frac{\partial}{\partial w_j} D_i u(y) = {} & D_k D_i u(y) \frac{\partial y_k}{\partial w_j} = c D_k D_i u(y) \, L_{kj} \\
\sum_j \frac{\partial^2}{\partial w_j^2} D_i u(y) = {} &  D_{\ell} D_k D_iu(y) \frac{\partial y_k}{\partial w_j} \frac{\partial y_{\ell}}{\partial w_j} =  c^2 D_{\ell} D_k D_iu(y) \, L_{kj} L_{\ell j}  = 0\\
\frac{\partial}{\partial w_j} D_i u(x) = {} &   D_j D_i u(x), \ \  \sum_j \frac{\partial^2}{\partial w_j^2} D_i u(x)=  D_j D_j D_i u(x)=0,
\end{split}
\]
where for the second line we used the fact that $\Delta_w y(0)=0$ and $ L_{kj} L_{\ell j} D_{\ell}D_k u = \Delta u=0$.
Hence, combining the above,
\[
\begin{split}
\Delta_w F > {} & 2  ( c D_k D_i u(y) L_{kj} - D_j D_i u(x)) ( c L_{ij} - \delta_{ij}) \\
= {} & 2 c^2 \Delta u (y) - 2c  L_{ki} D_k D_i u(y) -2c  L_{ij} D_j D_i u(x) + 2\Delta u (x) \\ 
= {} & - 2c L_{ki} D_k D_i u(y) -2c  L_{ij} D_j D_i u(x).
\end{split}
\]
Now we use the fact that $n$ is even, and we make an appropriate choice of $L$ following Lemma 4.1(a) of \cite{RR}.  Namely, after making an orthonormal change of coordinates, we may assume, without loss of generality that $Du(x_0)/|Du(x_0)|$ is $e_1$, and 
$$Du(y_0)/|Du(y_0)| = \cos \theta \, e_1 + \sin \theta \, e_2,$$
for some $\theta \in [0,2\pi)$. Here we are writing $e_1 = (1,0,\ldots 0)$ and $e_2 = (0,1,0, \ldots)$ etc for the standard unit basis vectors in $\mathbb{R}^n$.  Then define the isometry $L$ by 
$$L(e_i) = \left\{ \begin{array}{ll} \cos \theta \, e_i + \sin \theta \, e_{i+1}, & \quad \textrm{for } i=1,3, \ldots, n-1 \\
- \sin \theta \, e_{i-1} + \cos \theta \, e_{i}, & \quad \textrm{for } i=2, 4, \ldots, n.\end{array} \right.$$
In terms of entries of the matrix $(L_{ij})$, this means that $L_{kk} = \cos \theta$ for $k=1, \ldots n$ and for $\alpha = 1, 2, \ldots, n/2$, we have
$$L_{2\alpha-1, 2\alpha}= - \sin \theta, \quad L_{2\alpha, 2\alpha -1} = \sin \theta,$$
with all other entries zero.
Then 
\begin{equation} \label{Lki}
\begin{split}
\sum_{i,k} L_{ki} D_k D_i u(y) = {} & \sum_{k=1}^n L_{kk} D_k D_k u(y) + \sum_{\alpha=1}^{n/2} (L_{2\alpha-1, 2\alpha} + L_{2\alpha, 2\alpha -1}) D_{2\alpha-1} D_{2\alpha} u(y) \\ ={} & (\cos \theta) \Delta u(y)=0.
\end{split}
\end{equation}
Similarly 
$\sum_{i,k} L_{ki} D_k D_iu(x)=0.$
This completes the proof of Theorem \ref{theoremmain} in the case of $n$ even.

\bigskip

For $n$ odd, we argue in the same way as in \cite{RR}.  Let $L$ be an isometry of the even-dimensional $\mathbb{R}^{n+1}$, defined in the same way as above, but now with
$$L(Du(x_0), 0) = (c (Du)(y_0), 0).$$
 In Lemma \ref{lemmaRR}, replace $w \in \mathbb{R}^n$ by $w \in \mathbb{R}^{n+1}$.  Define $\pi :\mathbb{R}^{n+1} \rightarrow \mathbb{R}^n$ to be the projection $(w_1, \ldots, w_{n+1}) \mapsto (w_1, \ldots, w_n)$ and replace (\ref{RR}) and (\ref{eqnf}) by
 \begin{equation} \label{RR2}
u(x_0 + \pi(w)) = u\big(y_0 + c \pi(Lw) + f(w) \xi + \alpha(w) \xi\big), \end{equation}
where $\xi = \frac{Du(y_0)}{|Du(y_0)|}$ and  $f$ is given by
\begin{equation} \label{eqnf2}
f(w) = \frac{1}{|Du(y_0)|} (u(x_0+\pi (w)) - u(y_0+c\pi(Lw))).
\end{equation}
As in \cite{RR}, note that if $g: \mathbb{R}^n \rightarrow \mathbb{R}$ is harmonic in $\mathbb{R}^n$ then $w \mapsto g(\pi(Lw))$ is harmonic in $\mathbb{R}^{n+1}$.  In particular, $f$ is harmonic in a neighborhood of the origin in $\mathbb{R}^{n+1}$.  The function $G$ above becomes $G(w,\alpha) = u(y_0 + c\pi(Lw) + f(w) \xi + \alpha\xi)-u(x_0+\pi(w))$ with $w \in \mathbb{R}^{n+1}$, and we make similar changes to $F$.  It is straightforward to check that the rest of the proof goes through.

\begin{remark} The proof of Theorem \ref{theoremmain} also shows that when $\psi=0$ the quantity $Q(x,y)$ does not attain a strict interior \emph{minimum}.
\end{remark}

\section{Global to Infinitesimal}

Here we give a proof of Theorem \ref{theoremCMY} using the quantity $Q$.  
%From Theorem \ref{theoremmain} it follows that the level sets of $u$ are all strictly convex.  
We first claim that, for $x \in \Omega$ and $a>0$,
$$(Du(y) - Du(x)) \cdot (y-x) + a|y-x|^2 \le O(|y-x|^3), \quad \textrm{for } y \sim x, \ u(x)=u(y)$$
if and only if
$$(\kappa_1 |Du|) (x) \ge a.$$
Indeed, to see this, first choose coordinates  such that at $x$ we have $Du = (0, \ldots, 0, D_n u)$  and $(D_i D_j u)_{1 \le i,j \le n-1}$ is diagonal with $$D_1 D_1 u \ge \cdots \ge D_{n-1} D_{n-1} u.$$ 
For the ``if'' direction of the claim choose $y(t) = x + te_1 + O(t^2)$ such that $u(x)=u(y(t))$, for $t$ small.  By Taylor's Theorem,
$$(Du(y(t)) - Du(x)) \cdot (y(t)-x) + a |y(t)-x|^2= t^2 D_1 D_1 u(x) + a t^2 + O(t^3)$$
giving $D_1 D_1 u(x) \le -a$, which is the same as $|Du| \kappa_1 \ge a$.  Indeed from a well-known and elementary calculation (see for example \cite[Section 2]{CMY}),
 $$\kappa_1 = \frac{ - D_1 D_1u}{|Du|}$$
 at $x$. 
Hence $|Du| \kappa_1  \ge a$.  The ``only if'' direction of the claim follows similarly.

We will make use of this correspondence  in what follows.  

\begin{proof}[Proof of Theorem \ref{theoremCMY}]  By assumption, 
 $\kappa_1 |Du| \ge a>0$ on $\partial \Omega$.  It follows  from Theorem \ref{theoremmain} and the discussion above that the level sets of $u$ are all strictly convex.  Assume for a contradiction that $\kappa_1 |Du|$ achieves a strict (positive) minimum at a point $x_0$ in the interior of $\Omega$, say
\begin{equation} \label{x0}
(\kappa_1 |Du|) (x_0) = a- \eta >0 \ \textrm{ for some $\eta>0$}.
\end{equation}
We may assume without loss of generality that $\eta < a/6$.  Indeed, if not then if $x_0$ lies on the level set $\{ u= c\}$ for some $c \in (0,1)$ we can replace $\Omega$ by a convex ring $\{ c_0 < u < c_1 \}$  for $c_0, c_1$ with $0\le c_0 <c<c_1\le1$.   We still denote by $a$  the minimum value of $\kappa_1 |Du|$ on the  boundary of this new $\Omega$.  For appropriately chosen $c_0,c_1$ we have (\ref{x0}) and $\eta<a/6$. This changes the boundary conditions on $\partial \Omega_0$ and $\partial \Omega_1$ to $u=c_0$ and $u=c_1$, but this will not affect any of the arguments.

 Pick $\ve>0$ sufficiently small, so that the distance from $x_0$ to the boundary of $\Omega$ is much larger than $\ve$, and in addition, so that $\ve^{1/3} << \eta$.

Consider the quantity
$$Q(x,y) = (Du(y) - Du(x)) \cdot (y-x) + a|y-x|^2 - \frac{a}{6\ve^2} |y-x|^4,$$
and restrict to the set
$$\Sigma^{\ve} = \{ (x,y) \in \bar{\Omega} \times \bar{\Omega} \ | \ u(x)=u(y), \  |y- x| \le \ve\}.$$
Suppose that $Q$ attains a maximum  on $\Sigma^{\ve}$  at a point $(x, y)$.  First assume that $(x,y)$ lies in the boundary of $\Sigma^{\ve}$.  There are two possible cases:

\begin{enumerate}
\item If $x,y \in \Sigma^{\ve}$ with $x$ and $y$ in  $\partial \Omega$ (note that since $u(x)=u(y)$, if one of $x,y$ is a boundary point then so is the other) then since $\kappa_1 |Du| \ge a$ on $\partial \Omega$ we have
$$(Du(y) - Du(x)) \cdot (y-x) + a|y-x|^2 \le O(\ve^3).$$
Hence in this case $Q(x,y) \le O(\ve^3)$.

\smallskip

\item If $|y-x| = \ve$ then since $\kappa_1 |Du| \ge a -\eta$ everywhere,
$$Q(x,y) \le - (a - \eta) \ve^2 + O(\ve^3) + a \ve^2 - \frac{a}{6} \ve^2 = (\eta - \frac{a}{6}) \ve^2 + O(\ve^3) < 0,$$
by the assumption $\eta<a/6$.
\end{enumerate}

We claim that neither case can occur.  Indeed, consider  $y = x_0 + tv +O(t^2)$ for $t$ small, where $v$ is vector in the direction of the smallest curvature of the level set of $u$ and $x_0$ satisfies (\ref{x0}).  Then since $(|Du|\kappa_1)(x_0) = a -\eta$,
\[
\begin{split}
Q(x,y) = {} & - (a- \eta) |y-x_0|^2 + O(|y-x_0|^3) + a|y-x_0|^2 - \frac{a}{6\ve^2} |y-x_0|^4 \\
 = {} & \eta |y-x_0|^2 - \frac{a}{6\ve^2} |y-x_0|^4 +O(|y-x_0|^3).
 \end{split}
 \]
 If $|y-x_0| \sim \ve^{4/3}$ say then $Q(x_0,y) \sim \eta \ve^{8/3} + O(\ve^3) >> \ve^3$ since we assume $\eta >> \ve^{1/3}$.  
Since $Q$ here is larger  than in (1) or (2),  this rules out (1) or (2) as being possible cases for the maximum of $Q$.

This implies that $Q$ must attain an interior maximum, contradicting the 
argument of Theorem \ref{theoremmain}.  Here we use the fact that if $\psi(t) = at - \frac{a}{6\ve^2} t^2$  then for $t$ with $0 \le t \le \ve^2$,
$$\psi'(t) - 2| \psi''(t)| t = a (1  -   \frac{t}{\ve^2} )  \ge 0.$$
This completes the proof.
\end{proof}

\begin{remark} In \cite{CMY} and also \cite{MYY} it was shown that when $n=3$ the smallest principal curvature $\kappa_1$ also satisfies a minimum principle.  It would be interesting to know whether a modification of the quantity (\ref{Qd}) can give another proof of this.

\end{remark}

\end{document}